\documentclass[10pt, reqno]{amsart}
\usepackage{amssymb,amsmath,amsthm,color, cite}
\theoremstyle{plain}
\newtheorem{theorem}{Theorem}[section]

\newtheorem{proposition}[theorem]{Proposition}

\newtheorem{lemma}[theorem]{Lemma}
\newtheorem{corollary}[theorem]{Corollary}

\theoremstyle{remark}

\usepackage{graphicx}

\usepackage%{hyperref}
[bookmarks=true,
   bookmarksnumbered=false,
   bookmarkstype=toc]
{hyperref}
\numberwithin{equation}{section}

\newcommand{\C}{\mathbb{C}}
\newcommand{\R}{\mathbb{R}}

\renewcommand{\Im}{\operatorname{Im}}
\renewcommand{\Re}{\operatorname{Re}}

\def\({\left[}
\def\){\right]}

\newcommand{\eps}{\varepsilon}

\newcommand{\qtq}[1]{\quad\text{#1}\quad}

\begin{document}

\title[2$d$ cubic-quintic NLS]{Threshold scattering for the \\ 2d radial cubic-quintic NLS} 
\author{Jason Murphy}
\address{Department of Mathematics \& Statistics, Missouri S\&T}
\email{jason.murphy@mst.edu}

\begin{abstract} We consider the cubic-quintic nonlinear Schr\"odinger equation in two space dimensions.  For this model, X. Cheng established scattering for $H^1$ data with mass strictly below that of the ground state for the cubic NLS.  Subsequently, R. Carles and C. Sparber utilized the pseudoconformal energy estimate to obtain scattering at the sharp threshold for data belonging to a weighted Sobolev space. In this work, we remove the weighted assumption and establish scattering at the threshold for radial data in $H^1$.
\end{abstract}

\maketitle

%%%%%%%%%%%%%%%%
\section{Introduction}\label{S:intro}
We consider the cubic-quintic nonlinear Schr\"odinger equation (NLS) in two space dimensions with radial $H^1$ data:
\begin{equation}\label{nls}
\begin{cases} (i\partial_t + \Delta) u = -|u|^2u + |u|^4 u, \quad (t,x)\in\R\times\R^2, \\
u|_{t=0} = u_0\in H^1_{\text{rad}}(\R^2).
\end{cases}
\end{equation}

For initial data in $H^1$, one obtains local solutions to \eqref{nls} that conserve both the \emph{mass} and \emph{energy}, defined by  
\[
M(u) = \int |u|^2\,dx\qtq{and} E(u) = \int \tfrac12|\nabla u|^2 - \tfrac14|u|^4 + \tfrac16 |u|^6\,dx,
\]
respectively.  The questions of global well-posedness and long-time behavior for \eqref{nls} have been studied previously in \cite{Cheng, CarlesSparber} (see also \cite{CarlesSparber, KOPV, KOPV0, KMV, TVZ, Zhang} for related results in other dimensions).  In particular, the work \cite{Cheng} established scattering and $L_{t,x}^4$ spacetime bounds for initial data $u_0$ obeying the mass constraint
\begin{equation}\label{sub}
M(u_0)<M(Q),
\end{equation}
where $Q$ is the ground state for the cubic NLS, that is, the unique nonnegative, radial, decaying solution to\begin{equation}\label{def:Q}
-Q+\Delta Q+Q^3 = 0.
\end{equation}
As pointed out in \cite{Cheng, CarlesSparber}, this coincides with the sharp scattering threshold for the $2d$ cubic NLS 
\begin{equation}\label{cubic-NLS}
(i\partial_t+\Delta)w=-|w|^2 w
\end{equation}
(see \cite{Dodson, KTV}); furthermore, this threshold is also sharp for \eqref{nls}, in the sense that there exist nonlinear ground states at any mass strictly larger than that of $Q$ (cf. \cite[Theorem~1.9]{CarlesSparber}).

On the other hand, the authors of \cite{CarlesSparber} succeeded in establishing scattering for solutions obeying $M(u_0)=M(Q)$ under the additional assumption that  $xu_0\in L^2$.  With this assumption, one gains access to the pseudoconformal energy estimate, which has played an important role in the NLS scattering theory (see e.g. \cite{TsutsumiYajima, GinibreVelo}). Combining this estimate with rigidity results for the mass-critical NLS (as in \cite{Merle}), the authors of \cite{CarlesSparber} were able to obtain their result. 

In this paper, we establish scattering at the sharp mass threshold for radial data in $H^1$.  In particular, we are able to remove the weighted $L^2$ assumption appearing in \cite{CarlesSparber}.  Of course, the radial assumption also guarantees some degree of spatial localization, which continues to play an important role in this work.  The extension to non-radial $H^1$ data remains an interesting open problem and will be discussed briefly below. 

Our main result is the following.

\begin{theorem}[Threshold scattering]\label{T} Let $u_0\in H^1_{\text{rad}}(\R^2)$ satisfy 
\begin{equation}\label{mass}
M(u_0)=M(Q).
\end{equation}
Then the corresponding solution $u$ to \eqref{nls} is global in time and obeys 
\begin{equation}\label{stb}
\|u\|_{L_t^\infty H_x^1(\R\times\R^2)}+\|u\|_{L_{t,x}^4(\R\times\R^2)}<\infty.
\end{equation}
Consequently, $u$ scatters; that is, there exist $u_\pm\in H^1$ such that
\[
\lim_{t\to\pm\infty}\|u(t)-e^{it\Delta}u_\pm\|_{H^1(\R^2)}=0. 
\]
\end{theorem}

The proof of Theorem~\ref{T} boils down to establishing the $L_{t,x}^4$ estimate.  To see this, we first observe that the sharp Gagliardo--Nirenberg inequality in two dimensions may be written as follows (cf. \cite{Weinstein}): 
\[
\|u\|_{L^4(\R^2)}^4 \leq 2\tfrac{M(u)}{M(Q)}\|\nabla u\|_{L^2(\R^2)}^2. 
\]
Thus, for solutions obeying \eqref{mass}, we have
\begin{equation}\label{positive-energy}
E(u) \geq \tfrac16 \|u\|_{L^6}^6>0.
\end{equation}
In particular, solutions obey uniform bounds in $L^2$ and $L^6$ and hence in $L^4$.  This in turn yields uniform bounds in $\dot H^1$ by the conservation of energy. Consequently, by the local theory for \eqref{nls}, any solution obeying \eqref{mass} is global in time and uniformly bounded in $H^1$.  As standard arguments show that the bounds in \eqref{stb} imply scattering (see e.g. \cite{TVZ, Cheng}), we see that the proof of Theorem~\ref{T} indeed reduces to establishing the $L_{t,x}^4$ estimate in \eqref{stb}. 

Along with its clear connection to the works \cite{Cheng, CarlesSparber} and more generally to \cite{KOPV, KMV}, Theorem~\ref{T} may also be considered in the context of some recent works on threshold scattering for the $3d$ cubic NLS \cite{DLR, MMZ}.  In \cite{DLR}, the authors established a scattering result at the sharp threshold for the $3d$ focusing cubic NLS  outside of a convex obstacle, while \cite{MMZ} proved an analogous result in the presence of a repulsive external potential.  In both cases, the scattering threshold (or, more precisely, the threshold for obtaining uniform space-time bounds) coincides with that of the underlying cubic NLS on $\R^3$.  One finds that in the presence of a repulsive obstacle or potential, it is possible to obtain a scattering result even at the threshold.  This is in contrast to the case of the standard NLS or the NLS with an attractive potential, for which other threshold behaviors are possible (see e.g. \cite{DM, DR, YZZ}).  Our main result (as well as the scattering result of \cite{CarlesSparber}) is therefore analogous to the results of \cite{DLR, MMZ}, with the defocusing quintic term having a similar effect as that of a repulsive potential. 

In the remainder of the introduction, we will briefly discuss the strategy of the proof of Theorem~\ref{T}, as well as the possibility of extending Theorem~\ref{T} to the non-radial setting.  Without loss of generality, we consider scattering in the forward direction only.

\emph{Compactness for non-scattering solutions.}  In Section~\ref{S:compact}, we first show that if $u$ is a global, radial, $H^1$-bounded solution to \eqref{nls} with infinite $L_{t,x}^4$-norm on $[0,\infty)\times\R^2$, then the orbit of $u$ must be pre-compact in $L^2$ modulo some scale function (see Proposition~\ref{P:compact}). This step follows the familiar concentration-compactness approach, with the sub-threshold scattering result of \cite{Cheng} guaranteeing that the solution cannot split into multiple profiles.  As in the work of \cite{Cheng}, this step also utilizes an approximation of \eqref{nls} by \eqref{cubic-NLS} in the large scale limit.  The role of the radial assumption is to remove the need for a moving spatial center in the parametrization of the solution.

\emph{Energy evacuation.}  In Section~\ref{S:evacuation},  we utilize a virial/Morawetz estimate (essentially the same one used in \cite{ADM}) to prove an `energy evacuation' property for radial, $H^1$-bounded solutions to \eqref{nls} (see Proposition~\ref{P:virial}).  This does not rely on compactness in the sense of the previous step; rather, one simply relies on the radial assumption to obtain tightness in $L^4$ and $L^6$ via the radial Sobolev embedding.  This step shows that along some sequence of times $t_n\to\infty$, the solution must concentrate its energy in the kinetic energy component and increasingly far from the origin.  (In works such as \cite{ADM, DodMur}, this type of condition is used directly to derive scattering; however, this approach breaks down in the present setting due to the presence of the mass-critical nonlinearity in \eqref{nls}.) 

\emph{Energy localization.}  In Section~\ref{S:localization}, we prove a tightness property for the kinetic energy of compact solutions (see Proposition~\ref{P:localization}).  This does not follow immediately Proposition~\ref{P:compact}, as the compactness there is obtained only in the $L^2$ topology.  Here we follow the lead of \cite{KLVZ, LiZhang}, which established similar localization results in the radial mass-critical setting.  The main ingredients are a reduced Duhamel formula for compact solutions (Corollary~\ref{C:reduced-duhamel}) and an `in/out' decomposition for radial functions in $L^2$ (as introduced in \cite{KTV}), which together allow us to prove a quantitative frequency decay estimate (Lemma~\ref{L:frequency-decay}).  This part of the argument leans heavily on the radial assumption,  utilizing the in/out decomposition as well as the radial Sobolev embedding and a radial Strichartz estimate.

\emph{Conclusion.} In Section~\ref{S:conclusion}, we put together the pieces and prove the main result as follows:  The energy evacuation property (Proposition~\ref{P:virial}) shows that along a sequence $t_n\to\infty$, solutions have very little energy at bounded radii, while the energy localization property (Proposition~\ref{P:localization}) shows that compact solutions have very little energy at large radii (uniformly in time).  Combining these two properties and utilizing the conservation of energy, we can prove that compact solutions in fact have zero energy, which is incompatible with \eqref{positive-energy}.  Thus compact solutions cannot exist and so (by Proposition~\ref{P:compact}) Theorem~\ref{T} is proved.     

To close the introduction, let us point out that in contrast to \cite{CarlesSparber} (which also considered \eqref{nls}) and to \cite{DLR, MMZ} (which considered the $3d$ cubic NLS with repulsive obstacle/potential), it is not necessary in this work to incorporate any type of `modulation analysis', that is, to obtain a precise description of the solution as it approaches the orbit of $Q$.  We do, of course, rely on the sharp Gagliardo--Nirenberg inequality, but we do not need any refinements thereof.  The need for such analysis is ultimately obviated by the fact that we can obtain kinetic energy localization solely by relying on the radial assumption.   The reader will notice, for example, that this part of the argument is completely insensitive to the combination of signs in the nonlinearity.  In the non-radial case, on the other hand, it is less clear how one might obtain compactness at the $\dot H^1$ level.  It seems likely that some kind of modulation analysis may again play an important role.

%%%%%%%%%%%%%%%%%
\section{Notation and preliminaries}\label{S:notation}

We write $A\lesssim B$ to denote the inequality $A\leq CB$ for some $C>0$.  We utilize the standard Lebesgue and Sobolev spaces, as well as the $L_t^q L_x^r$ notation for mixed Lebesgue space-time norms.  We denote the usual Littlewood--Paley frequency projections by $P_N$, $P_{\leq N}$, $P_{>N}$, and so on, and we denote the `fattened' operators by $\tilde P_N$, so that (for example) $\tilde P_N P_N = P_N$.  We also make use of the standard Bernstein estimates associated to these operators.

Throughout the paper, we use the following radial Sobolev embedding estimate, which may be proved by the fundamental theorem of calculus and Cauchy--Schwarz (see also \cite{Strauss}):
\[
\| |x|^{\frac12}u\|_{L^\infty(\R^2)} \lesssim \|u\|_{L^2(\R^2)}^{\frac12}\|u\|_{\dot H^1(\R^2)}^{\frac12}\qtq{for}u\in H^1_{\text{rad}}(\R^2).
\]
Here and throughout we use the subscript ${}_{\text{rad}}$ to emphasize the restriction to radial functions. 

We use the notation $\chi_R$ to denote a smooth cutoff to the set $\{|x|\leq R\}$. To save space in formulas, we sometimes use the notation
\[
\chi_R^c(x):=1-\chi_R(x),
\]
where ${}^c$ is meant to connote `complement'.

\subsection{Local theory} In this section we briefly review the local theory for \eqref{nls} (see e.g. \cite{Cheng} for more details).  

The equation \eqref{nls} admits local solutions for any initial data in $H^1$, and (as described in the introduction) we obtain global existence and $H^1$ bounds for solutions obeying $M(u)\leq M(Q)$.  We may also construct solutions scattering to prescribed asymptotic states as $t\to\pm\infty$.  More generally, we obtain that any $H^1$ bounded solution belonging to $L_{t,x}^4(\R\times\R^2)$ scatters in both time directions in $H^1$.

We will also make use of the following stability result for \eqref{nls} (see e.g. Proposition~3.3 in \cite{Cheng}).

\begin{proposition}[Stability]\label{P:stab} Suppose $w:I\times\R^2\to\C$ solves
\[
(i\partial_t + \Delta)w = -|w|^2 w + |w|^4 w + e,\quad w(t_0)=w_0
\]
for some function $e:I\times\R^2\to\C$.  Let $u_0\in H^1$ and suppose that
\[
\|w_0\|_{H^1}+\|u_0\|_{H^1} \leq E \qtq{and} \|\langle \nabla\rangle^{\frac12} w\|_{L_{t,x}^4(I\times\R^2)} \leq L. 
\]
There exists $\eps_0=\eps_0(E,L)>0$ so that if $0<\eps<\eps_0$ and
\[
\|w_0-u_0\|_{H^1} + \|\langle \nabla\rangle^{\frac12}e\|_{L_{t,x}^{\frac43}(I\times\R^2)} < \eps,
\]
then there exists a solution $u:I\times\R^2\to\C$ to \eqref{nls} with $u(t_0)=u_0$ satisfying
\[
\|\langle\nabla\rangle^{\frac12}[u-w]\|_{L_{t,x}^4(I\times\R^2)}\lesssim_{E,L}\eps \qtq{and}\|\langle \nabla\rangle^{\frac12}u\|_{L_{t,x}^4(I\times\R^2)}\lesssim_{E,L} 1. 
\]
\end{proposition}

%%%%%%
\subsection{Concentration-compactness}

We record here a linear profile decomposition adapted to the $L^2(\R^2)\to L_{t,x}^4$ Strichartz estimate for $e^{it\Delta}$.  Such a result was established originally in \cite{MerleVega} for the two-dimensional case.  We will need the following decomposition for radial, $H^1$-bounded sequences:

\begin{proposition}[Linear profile decomposition]\label{P:LPD} Let $u_n$ be a bounded sequence in $H^1_{\text{rad}}$.  Then the following holds up to a subsequence:

There exist $J^*\in\{0,1,2,\dots,\infty\}$, non-zero profiles $\{\phi^j\}_{j=1}^{J^*}\subset L^2_{\text{rad}}$, and parameters $(t_n^j,\lambda_n^j)$  satisfying the following:

For each finite $0\leq J\leq J^*$, we have
\[
u_n = \sum_{j=1}^J \phi_n^j + r_n^J,\qtq{where} \phi_n^j = e^{it_n^j\Delta}\bigl\{\tfrac{1}{\lambda_n^j}(P_n^j\phi^j)(\tfrac{\cdot}{\lambda_n^j})\bigr\},
\]
with $\lambda_n^j\equiv 1$ or $\lambda_n^j\to \infty$, $t_n^j \equiv 0$ or $(\lambda_n^j)^{-2} t_n^j\to\pm\infty$, and
\[
P_n^j := \begin{cases} P_{\leq [\lambda_n^j]^\theta} & \lambda_n^j \to\infty \\ \text{Id} & \lambda_n^j\equiv 1\end{cases}
\]
for some $\theta\in(0,1)$.  In addition, if $\lambda_n^j\equiv 1$ then $\phi^j\in H^1$.

For each finite $0\leq J\leq J^*$, we have the following decoupling properties:
\begin{align*}
&\lim_{n\to\infty} \bigl\{ \| u_n\|_{\dot H^s}^2 - \sum_{j=1}^J \| \phi_n^j\|_{\dot H^s}^2 - \| r_n^J\|_{\dot H^s}^2\bigr\} = 0,\quad s\in\{0,1\}, \\
& \lim_{n\to\infty} \bigl\{ \|f_n\|_{L^p}^p - \sum_{j=1}^J \|\phi_n^j\|_{L^p}^p - \|r_n^J\|_{L^p}^p \bigr\} = 0,\quad p\in\{4,6\}. 
\end{align*}

The remainder obeys
\[
\limsup_{J\to J^*}\limsup_{n\to\infty} \|\langle \nabla\rangle^{\frac12} e^{it\Delta} r_n^J \|_{L_{t,x}^4(\R\times\R^2)}=0.
\]

Finally, the parameters $(t_n^j,\lambda_n^j)$ are asymptotically orthogonal in the following sense: for any $j\neq k$, 
\[
\lim_{n\to\infty} \bigl\{ |\log\tfrac{\lambda_n^j}{\lambda_n^k}| + \tfrac{|t_n^j-t_n^k|}{(\lambda_n^j)^2}\bigr\} = \infty. 
\]
\end{proposition}

A similar result appears in \cite[Theorem~4.2]{Cheng} without the radial assumption.  In that setting, one must contend with a more complicated group of symmetries, including spatial translations $x_n^j$ and boosts $\xi_n^j$.  In the radial setting, these parameters may be taken to be identically zero, and the profiles themselves may be taken to be radial functions (although this last point is less essential for our purposes).  To see this, one may follow the arguments presented in \cite[Section~7]{TVZ-mass}, particularly the proof of Theorem~7.3 therein. 

%%%%%%%%%%%%%%%%%
\section{Compactness for non-scattering solutions}\label{S:compact}

In this section, we show that if $u$ is a radial solution obeying \eqref{mass}, but $u$ fails to scatter, then $u$ must exhibit some compactness in $L^2$. 

\begin{proposition}[Compactness]\label{P:compact}  Suppose $u$ is a radial, forward-global solution to \eqref{nls} obeying $M(u)=M(Q)$,
\[
\|u\|_{L_t^\infty H_x^1([0,\infty)\times\R^2)}\lesssim 1 ,\qtq{and} \|u\|_{L_{t,x}^4([0,\infty)\times\R^2)}=\infty. 
\]
Then there exists $\lambda:[0,\infty)\to[1,\infty)$ such that
\begin{equation}\label{compact}
\{\lambda(t)u(t,\lambda(t)x):t\in[0,\infty)\} \qtq{is pre-compact in}L^2(\R^2). 
\end{equation}
\end{proposition}

\begin{proof} The argument follows the usual concentration-compactness approach and is similar in structure to the arguments appearing in \cite{MMZ, DLR}, borrowing some of the main ideas from \cite{Cheng}, as well.  Accordingly, we will give a fairly abbreviated presentation.  

The essential point is to show that for any $\tau_n\to\infty$, we may find a subsequence in $n$ and parameters $\lambda_n\geq 1$ so that $\lambda_n u(\tau_n,\lambda_n x)$ converges strongly in $L^2$; we will briefly discuss why this implies the existence of the scale function $\lambda(t)$ below.  

To begin, we apply the linear profile decomposition (Proposition~\ref{P:LPD}) to obtain the following for any $0\leq J\leq J^*$:
\[
u_n:=u(\tau_n) = \sum_{j=1}^J \phi_n^j + r_n^J,\quad \phi_n^j = e^{it_n^j\Delta}\bigl\{ \tfrac{1}{\lambda_n^j}(P_n^j\phi)(\tfrac{\cdot}{\lambda_n^j})\bigr\}
\]
We consider three possibilities: $J^*=0$ (vanishing), $J^1=1$ (compactness), or $J^*\geq 2$ (dichotomy). 

If $J^*=0$, then we obtain
\[
\lim_{n\to\infty} \|e^{it\Delta} u(\tau_n)\|_{L_{t,x}^4([0,\infty)\times\R^2)}=0.
\]
Applying the stability result (Proposition~\ref{P:stab}), we find that
\[
\|u(t+\tau_n)\|_{L_{t,x}^4((0,\infty)\times\R^2)} = \|u\|_{L_{t,x}^4((\tau_n,\infty)\times\R^2)} \lesssim 1
\]
for all large $n$, which yields a contradiction.  In particular, `vanishing' does not occur.

We next suppose that $J^*\geq 2$ and again seek a contradiction.  We will use the profiles $\phi_n^j$ to build approximate solutions to \eqref{nls}.  We observe that by the mass decoupling, each $\phi_n^j$ satisfies the sub-threshold mass condition \eqref{sub}. 

If $\lambda_n^j\equiv 1$ and $t_n^j \equiv 0$, we let $v^j$ be the solution to \eqref{nls} with initial data $v^j(0)=\phi^j$.  By the results of \cite{Cheng}, this solution scatters and obeys $L_{t,x}^4$ spacetime bounds.  If $\lambda_n^j\equiv 1$ and $t_n^j\to\pm\infty$, we let $v^j$ be the solution to \eqref{nls} satisfying 
\[
\lim_{t\to\pm\infty}\|v^j-e^{it\Delta}\phi^j\|_{H^1} = 0.
\]
In either case, we then take $v_n^j(t,x)=v^j(t+t_n^j,x)$. 

If instead $\lambda_n^j\to\infty$, then we will construct a scattering solution $v_n^j$ to \eqref{nls} with $v_n^j(0)=\phi_n^j$ by approximating with a solution to \eqref{cubic-NLS} (as in \cite[Theorem~5.1]{Cheng}).  For the sake of completeness (and because this step is less standard than the rest of the argument), we provide the proof here. 

\begin{lemma}[Large scale approximation by \eqref{cubic-NLS}]\label{L:approximation}  Suppose that $M(\phi)<M(Q)$, $\lambda_n\to\infty$, and either $t_n\equiv 0$ or $\lambda_n^{-2}t_n\to\pm\infty$.  Write $P_n=P_{\leq\lambda_n^\theta}$ for some $0<\theta<1$.  Then for all $n$ sufficiently large, there exists a global, scattering solution $v_n$ to \eqref{nls} satisfying
\[
v_n(0)=\phi_n := e^{it_n\Delta}\{\tfrac{1}{\lambda_n}(P_n\phi)(\tfrac{\cdot}{\lambda_n})\}.
\]
\end{lemma}

\begin{proof} If $t_n\equiv 0$, we let $w_n$ be the solution to the cubic NLS \eqref{cubic-NLS} with $w_n(0)=P_n\phi$.  If instead $\lambda_n^{-2}t_n\to\pm\infty$, we let $w_n$ be the solution to \eqref{cubic-NLS} satisfying
\[
\|w_n-e^{it\Delta}P_n\phi\|_{L^2}\to 0 \qtq{as}t\to\pm\infty. 
\]
We observe that by the main result of \cite{Dodson} and persistence of regularity for \eqref{cubic-NLS}, we have
\[
\|w_n\|_{L_t^\infty L_x^2\cap L_{t,x}^4(\R\times\R^2)} \lesssim 1\qtq{and} \|| \nabla|^{s} w_n\|_{ L_t^\infty L_x^2\cap L_{t,x}^4(\R\times\R^2)} \lesssim \lambda_n^{s\theta} 
\]
for all $n$ large and $s\in[0,1]$.

We now define approximate solutions to \eqref{nls} via
\[
\tilde v_n(t,x)=\lambda_n^{-1} w_n(\lambda_n^{-2}t,\lambda_n^{-1}x),
\]
which solve
\[
(i\partial_t+\Delta)\tilde v_n + |\tilde v_n|^2 \tilde v_n - |\tilde v_n|^4 \tilde v_n = -\lambda_n^{-5} (|w_n|^4 w_n)(\lambda_n^{-2}t,\lambda_n^{-1}x) 
\]
and obey
\[
\| |\nabla|^s \tilde v_n\|_{L_t^\infty L_x^2\cap L_{t,x}^4(\R\times\R^2)} \lesssim \lambda^{-s(1-\theta)}
\]
for all $n$ large and $s\in[0,1]$.  Similarly, we may estimate the error on $\R\times\R^2$ as follows: 
\begin{align*}
\| & |\nabla|^{s}\bigl\{\lambda_n^{-5}[|w_n|^4w_n](\lambda_n^{-2}t,\lambda_n^{-1}x)\bigr\}\|_{L_{t,x}^{\frac43}} \\
& \lesssim \lambda_n^{-5}\| w_n(\lambda_n^{-2}t,\lambda_n^{-1} x)\|_{L_{t,x}^8}^4 \| |\nabla|^s[w_n(\lambda_n^{-2}t,\lambda_n^{-1}x)]\|_{L_{t,x}^4} \\
& \lesssim \lambda_n^{-2-s}\|w_n\|_{L_{t,x}^8}^4 \| |\nabla|^s w_n\|_{L_{t,x}^4}  \lesssim \lambda_n^{-(2+s)(1-\theta)}
\end{align*}
for $s\in\{0,\tfrac12\}$, so that
\[
\|\langle \nabla\rangle^{\frac12}\bigl\{ \lambda_n^{-5}[|w_n|^4w_n](\lambda_n^{-2}t,\lambda_n^{-1}x)\bigr\}\|_{L_{t,x}^{\frac43}(\R\times\R^2)} \to 0 \qtq{as}n\to\infty. 
\] 

We next observe that the approximate solutions $\tilde v_n$ agree with $\phi_n$ in $H^1$ at time $t=t_n$. Indeed, if $t_n\equiv 0$, we have $\tilde v_n(t_n)=\phi_n$, and hence it suffices to consider the case $\lambda_n^{-2}t_n\to\pm\infty$.  In this case, we estimate
\begin{align*}
\| &\tilde v_n(t_n)-\lambda_n^{-1}\{e^{it_n\Delta}(P_n\phi)(\tfrac{\cdot}{\lambda_n})\}\|_{H^1} \\
& \lesssim \|w_n(t_n)-e^{it_n\Delta}P_n\phi\|_{L^2}+ \lambda_n^{-1}\{\|w_n\|_{L_t^\infty \dot H_x^1}+ \| P_n\phi\|_{\dot H^1}\}, 
\end{align*}
which tends to zero as $n\to\infty$. 

Applying the stability result (Proposition~\ref{P:stab}), we therefore deduce that there exist true solutions to \eqref{nls} with $u_n(0)=\varphi_n$ obeying 
\[
\|u_n\|_{L_t^\infty H_x^1(\R\times\R^2)} \lesssim 1\qtq{and}\| \langle \nabla\rangle^{\frac12} u_n\|_{L_{t,x}^4(\R\times\R^2)}\lesssim 1
\]
for all $n$ large. \end{proof}

Returning to the proof of Proposition~\ref{P:compact}, we now define approximate solutions $u_n^J$ to \eqref{nls} by
\[
u_n^J(t) = \sum_{j=1}^J v_n^j(t) + e^{it\Delta} r_n^J. 
\]
Proceeding as in \cite{Cheng}, we may now utilize the orthogonality of the parameters $(t_n^j,\lambda_n^j)$ and the vanishing of the $L_{t,x}^4$-norm of $e^{it\Delta}r_n^J$ to verify that (i) $u_n^J$ obey global space-time bounds; (ii) the $u_n^J$ asymptotically agree with $u_n=u(\tau_n)$ in $H^1$ a $n,J\to\infty$; and (iii) the $u_n^J$ define good approximate solutions to \eqref{nls} in the sense required by Proposition~\ref{P:stab}.  Consequently, we deduce that 
\[
\|u(t+\tau_n)\|_{L_{t,x}^5([0,\infty)\times\R^2)} = \|u\|_{L_{t,x}^5([\tau_n,\infty)\times\R^2)} \lesssim 1
\]
uniformly for all large $n$, which yields a contradiction.  We conclude that `dichotomy' does not occur.

We are therefore left with the decomposition
\[
u_n=u(\tau_n) = e^{it_n\Delta}\{\tfrac{1}{\lambda_n}(P_n \phi)(\tfrac{\cdot}{\lambda_n})\} + r_n,
\]
with $r_n\rightharpoonup 0$ weakly in $L^2$.  By the mass decoupling property, if $r_n$ were not to converge to zero strongly in $L^2$, then we could apply the same argument we used to preclude dichotomy to obtain a contradiction once again.  We may also preclude the possibility that $\lambda_n^{-2}t_n\to\pm\infty$, since in this case we would obtain
\[
\|e^{it\Delta}u_n\|_{L_{t,x}^4((-\infty,0)\times\R^2)}\to 0 \qtq{or}\|e^{it\Delta}u_n\|_{L_{t,x}^4((0,\infty)\times\R^2)}\to 0,
\]
respectively. Applying Proposition~\ref{P:stab}, we would therefore obtain bounds on either $(\tau_n,\infty)$ or $(-\infty,\tau_n)$ for large $n$, which yields a contradiction in either case.  Noting that $P_n\to \text{Id}$ strongly in $L^2$, we finally deduce that
\[
u(\tau_n) = \tfrac{1}{\lambda_n}\phi(\tfrac{\cdot}{\lambda_n})+o(1) \qtq{in} L^2,
\]
as desired. 

To finish the proof, let us briefly discuss how to deduce the existence of a function $\lambda(t)$ so that \eqref{compact} holds. To begin, we claim that there exist $C,c>0$ so that
\[
\sup_{\lambda_0\geq 1} \int_{|x|\leq C}\lambda_0^2|u(t,\lambda_0x)|^2\,dx \geq c>0 \qtq{for all}t\geq 0. 
\]
If not, we may find $C_n\to\infty$ and $\{t_n\}$ such that this supremum tends to zero.  However, writing 
\[
u(t_n)=\tfrac{1}{\lambda_n}\phi(\tfrac{\cdot}{\lambda_n}) + o(1) \qtq{in} L^2
\]
along a subsequence, we obtain 
\[
\int_{|x|\leq C_n} |\phi(t,x)|^2\,dx = o(1) \qtq{as}n\to\infty,
\]
yielding the contradiction $\phi=0$.

We may therefore define $\lambda(t)\geq 1$ and find $C,c>0$ so that 
\begin{equation}\label{ntlb}
\int_{|x|\leq C} \lambda(t)^2|u(t,\lambda(t)x)|^2\,dx\geq c>0 \qtq{for all}t\geq 0. 
\end{equation}
We claim that \eqref{compact} holds for this choice of $\lambda(t)$.  To see this, we take an arbitrary squence $\{t_n\}\subset[0,\infty)$ and obtain
\[
u(t_n)=\tfrac{1}{\lambda_n}\phi(\tfrac{\cdot}{\lambda_n})+o(1),\qtq{so that} \lambda(t_n) u(t_n,\lambda(t_n))=\tfrac{\lambda(t_n)}{\lambda_n}\phi(\tfrac{\lambda(t_n)x}{\lambda_n})+o(1)
\]
in $L^2$ along a subsequence. Finally, we observe that the sequence
\[
\tfrac{\lambda(t_n)}{\lambda_n}\phi(\tfrac{\lambda(t_n)x}{\lambda_n})
\]
either converges strongly in $L^2$ or converges weakly to zero in $L^2$ along some further subsequence.  As weak convergence to zero is incompatible with \eqref{ntlb}, we conclude the proof. \end{proof}

With Proposition~\ref{P:compact}, we turn to a few basic properties of `compact' solutions to \eqref{nls}.  We begin with the following upper bound for the scale function $\lambda(t)$.  This bound is typical of almost periodic solutions to NLS (see e.g. \cite[Corollary~5.19]{KVClay}), but requires a different proof in the present setting due to the broken scaling symmetry.

\begin{lemma}\label{L:bound-lambda} Suppose $u$ is a solution as in Proposition~\ref{P:compact}.  Then there exists $C>0$ and $T_0\geq 1$ such that
\[
\lambda(t) \leq C t^{\frac12} \qtq{for all}t>T_0. 
\]
\end{lemma}

\begin{proof} We argue by contradiction.  If the lemma fails, then we may find $t_n\to\infty$ and $C_n\to\infty$ such that
\[
\lambda_n:=\lambda(t_n)\geq C_n t_n^{\frac12}. 
\]
Passing to a subsequence, we obtain $v_0\in L^2$ so that
\[
\lambda_nu(t_n,\lambda_n x)\to v_0 \qtq{in} L^2. 
\]

We now let $w$ denote the maximal-lifespan solution to the standard cubic NLS \eqref{cubic-NLS} with initial data $v_0$.  By the local theory for \eqref{cubic-NLS}, we may find $\delta=\delta(v_0)>0$ sufficiently small that
\[
\|w\|_{L_{t,x}^4([-\delta,0]\times\R^2)}\lesssim 1. 
\]
Similarly, we let $w_n$ denote the maximal-lifespan solution to \eqref{cubic-NLS} with initial data
\[
w_n(0)=P_{\leq \lambda_n^\theta} v_0
\]
for some $\theta\in(0,1)$. Observing that $w_n(0)\to v_0$ strongly in $L^2$ and that
\[
\|\langle \nabla\rangle^{\frac12} w_n(0)\|_{L^2}\lesssim \lambda_n^{\frac\theta2}\|v_0\|_{L^2},
\]
we have (by the stability theory and persistence of regularity for \eqref{cubic-NLS}) that 
\[
\|\langle\nabla \rangle^{\frac12} w_n\|_{L_{t,x}^4([-\delta,0]\times\R^2)}+\|\langle\nabla\rangle^{\frac12}w_n\|_{L_t^8 L_x^{\frac83}([-\delta,0]\times\R^2)}\lesssim \lambda_n^{\frac\theta2} 
\]
for sufficiently large $n$. 

We now define approximate solutions to \eqref{nls} by
\[
\tilde u_n(t,x)=\lambda_n^{-1} w_n(\lambda_n^{-2}t,\lambda_n^{-1}x), 
\]
which satisfy
\[
(i\partial_t +\Delta)\tilde u_n +|\tilde u_n|^2 \tilde u_n - |\tilde u_n|^4 \tilde u_n = \lambda_n^{-5}[|w_n|^4w_n](\lambda_n^{-2}t,\lambda_n^{-1}x) 
\]
on the intervals
\[
I_n := [-\delta\lambda_n^2,0]\supset[-t_n,0] \qtq{for sufficiently large}n. 
\]

By construction, a change of variables, and dominated convergence, we have
\[
\| \tilde u_n(0)-u(t_n)]\|_{L^2} = \|P_{\leq\lambda_n^\theta}v_0(x)-\lambda_n u(t_n,\lambda_nx))\|_{L^2}\to 0\qtq{as}n\to\infty.
\]
On the other hand, by Bernstein's inequality,
\[
\|\tilde u_n(0)-u(t_n)\|_{\dot H^1} \leq \|\tilde u_n(0)\|_{\dot H^1} + \|u\|_{L_t^\infty \dot H_x^1} \lesssim \lambda_n^{-1+\theta}+1\lesssim 1.
\]
Thus, by interpolation we have
\[
\lim_{n\to\infty} \|\langle \nabla\rangle^{\frac12}[\tilde u_n(0)-u(t_n)]\|_{L^2} = 0. 
\]

Next, we observe
\[
\||\nabla|^s \tilde u_n\|_{L_{t,x}^4(I_n\times\R^2)} \lesssim \lambda_n^{-s}\||\nabla|^s w_n\|_{L_{t,x}^4([-\delta,0]\times\R^2)} \lesssim \lambda_n^{-s(1-\theta)}
\]
for $s\in\{0,\tfrac12\}$, so that 
\[
\|\langle \nabla\rangle^{\frac12} \tilde u_n\|_{L_{t,x}^4(I_n\times\R^2)}\lesssim 1\qtq{for large}n.
\]

Similarly, we estimate the error by
\begin{align*}
\| & |\nabla|^{s}\bigl\{\lambda_n^{-5}[|w_n|^4w_n](\lambda_n^{-2}t,\lambda_n^{-1}x)\bigr\}\|_{L_{t,x}^{\frac43}(I_n\times\R^2)} \\
& \lesssim \lambda_n^{-5}\| w_n(\lambda_n^{-2}t,\lambda_n^{-1} x)\|_{L_{t,x}^8(I_n\times\R^2)}^4 \| |\nabla|^s[w_n(\lambda_n^{-2}t,\lambda_n^{-1}x)]\|_{L_{t,x}^4(I_n\times\R^2)} \\
& \lesssim \lambda_n^{-2-s}\|w_n\|_{L_{t,x}^8([-\delta,0]\times\R^2)}^4 \| |\nabla|^s w_n\|_{L_{t,x}^4([-\delta,0]\times\R^2)}  \lesssim \lambda_n^{-(2+s)(1-\theta)}
\end{align*}
for $s\in\{0,\tfrac12\}$, so that
\[
\|\langle \nabla\rangle^{\frac12}\bigl\{ \lambda_n^{-5}[|w_n|^4w_n](\lambda_n^{-2}t,\lambda_n^{-1}x)\bigr\}\|_{L_{t,x}^{\frac43}(I_n\times\R^2)} \to 0 \qtq{as}n\to\infty. 
\] 

Thus, applying the stability result (Proposition~\ref{P:stab}), we deduce that
\[
\|u(t+t_n)\|_{L_{t,x}^4([-t_n,0]\times\R^2)} = \|u\|_{L_{t,x}^4([0,t_n]\times\R^2)} \lesssim 1
\]
for all $n$ large, which yields a contradiction.  \end{proof}

With the bound for $\lambda(t)$ in place, we can establish the following `reduced' Duhamel formula for compact solutions to \eqref{nls}. 

\begin{corollary}[Reduced Duhamel formula]\label{C:reduced-duhamel}  Let $u$ be a solution to \eqref{nls} as in Proposition~\ref{P:compact}. For $t\geq 0$, the following holds as a weak limit in $L^2$:
\[
u(t) = \lim_{T\to\infty} i\int_t^T e^{i(t-s)\Delta}F(u(s))\,ds.
\]
\end{corollary} 

\begin{proof} We argue essentially as in the proof of \cite[Proposition~5.23]{KVClay}.  By the Duhamel formula and $L^2$-boundedness, it suffices to show that
\[
\lim_{t\to\infty}\langle u(t), e^{it\Delta}\phi\rangle = 0 \qtq{for all}\phi\in C_c^\infty(\R^2). 
\]
Fix $\phi \in C_c^\infty(\R^2)$ and $\eps>0$.  We first choose $C_\eps>0$ large enough that
\[
\int_{|x|>C_\eps\lambda(t)}|u(t,x)|^2\,dx < \eps \qtq{for all}t\in[0,\infty).
\]
Applying Cauchy--Schwarz, Lemma~\ref{L:bound-lambda}, and the dispersive estimate, we obtain \begin{align*}
|\langle u(t),e^{it\Delta}\phi\rangle|^2 & \lesssim \|u\|_{L_t^\infty L_x^2}^2\cdot \int_{|x|\leq C_\eps\lambda(t)} |e^{it\Delta}\phi|^2\,dx + \eps\|\phi\|_{L^2}^2 \\
& \lesssim t^{-2}[C_\eps\lambda(t)]^2\|\phi\|_{L^1}^2 + \eps \lesssim C_\eps^2 t^{-1} + \eps
\end{align*}
for all $t$ sufficiently large, which yields the result.  \end{proof}

%%%%%%%%%%%%%%%%%
\section{Energy evacuation}\label{S:evacuation}

In this section, we prove the `energy evacuation' property as described in the introduction.  The main ingredient is a virial/Morawetz estimate, which is essentially contained already in \cite{ADM}.  We remark that this estimate does not require any `compactness' for the solution, beyond the tightness in $L^4$ and $L^6$ afforded by the radial assumption via the radial Sobolev embedding inequality. The precise result we prove is the following.

\begin{proposition}\label{P:virial} Let $u$ be a radial, $H^1$-bounded solution to \eqref{nls}. Then for any $T,R\geq$ we have
\begin{equation}\label{MV}
\int_0^T\int \bigl\{\chi_R|\nabla u|^2 - \tfrac12 |u|^4 + \tfrac23|u|^6\bigr\} dx\,dt \lesssim R+R^{-1}T,
\end{equation}
where $\chi_R$ is a smooth cutoff to $\{|x|\leq R\}$.  Consequently, for any $C>0$, there exists $t_n\to\infty$ such that
\begin{equation}\label{evacuation}
\limsup_{n\to\infty} \int \tfrac12 \chi_{C\lambda(t_n)} |\nabla u(t_n,x)|^2 - \tfrac14|u(t_n,x)|^4 + |u(t_n,x)|^6 \,dx \leq 0.
\end{equation}
\end{proposition}

\begin{proof} We argue as in the proof of \cite[Proposition~3.1]{ADM}.  We begin by letting $\phi$ be a smooth radial function satisfying
\[
\phi(x) = \begin{cases} 1 & 0 \leq |x|\leq 1, \\ 0 & |x|>2,\end{cases}
\]
and we denote $\phi=\phi(r)$, where $r=|x|$.  We also use $'$ or $\partial_r$ to denote radial derivatives.

We next define
\[
\psi(x)=\tfrac{1}{|x|}\int_0^{|x|}\phi(\rho)\,d\rho,
\]
so that $\psi(r)=\phi(r)$ for $r\leq 1$ and
\begin{equation}\label{weight-ids}
|\psi(x)|\lesssim \min\{1,|x|^{-1}\}\qtq{and} r\phi'(r)=\phi(r)-\psi(r). 
\end{equation}
In particular, we find $\psi'(r)=\phi'(r)=0$ for $r\leq 1$, and
\begin{equation}\label{weight-ids2}
|\psi'(x)|\lesssim |x|^{-2} \qtq{for}|x|>1. 
\end{equation}

We now define
\[
A(t) = \int \psi(\tfrac{x}{R})x\cdot\Im[\bar u\nabla u]\,dx,
\]
which obeys
\[
\|A(t)\|_{L_t^\infty([0,\infty))}\lesssim R\|u\|_{L_t^\infty H_x^1}^2 \lesssim R. 
\]

Next, we use the equation \eqref{nls} to compute
\begin{align}
\tfrac{dA}{dt} & = \Re\int \psi(\tfrac{x}{R})x_k[\bar u u_{jjk}-\bar u_{jj}u_k]\,dx \label{Mor1} \\
& \quad + \Re\int \psi(\tfrac{x}{R})x_k[\bar u \partial_k (|u|^2 u)-|u|^2\bar u u_k]\,dx \label{Mor2}\\
& \quad - \Re\int \psi(\tfrac{x}{R})x_k[\bar u \partial_k(|u|^4 u)-|u|^4 \bar u u_k]\,dx.\label{Mor3} 
\end{align}

We first observe
\[
\Re[\bar u u_{jjk}-\bar u_{jj} u_k] = \tfrac12 \partial_{jjk}|u|^2 - 2\Re\partial_j[\bar u_j u_k].
\]
Using this together with \eqref{weight-ids}, we obtain
\[
\eqref{Mor1} = -\tfrac12\int \Delta[\psi(\tfrac{x}{R})+\phi(\tfrac{x}{R})]|u|^2\,dx + 2\int \psi(\tfrac{x}{R})|\nabla u|^2 + \psi'(\tfrac{x}{R})\tfrac{|x|}{R}|\partial_r u|^2\,dx. 
\]
Using
\[
\Delta[\psi(\tfrac{x}{R})+\phi(\tfrac{x}{R})]=\tfrac{1}{R^2}\phi''(\tfrac{x}{R})+\tfrac{1}{R|x|}[2\phi'(\tfrac{x}{R})-\psi'(\tfrac{x}{R})]
\]
and \eqref{weight-ids}--\eqref{weight-ids2}, we obtain 
\[
\eqref{Mor1} = 2\int \phi(\tfrac{x}{R})|\nabla u|^2\,dx + \mathcal{O}(R^{-2}\|u\|_{L^2}^2). 
\] 

We next observe the general identity
\[
\Re\{\bar u\partial_k(|u|^p u) - |u|^p \bar u u_k\} = \tfrac{p}{p+2}\partial_k(|u|^{p+2}) \qtq{for}p>0,
\]
which implies
\[
\eqref{Mor2}+\eqref{Mor3} = -\int[\psi(\tfrac{x}{R})+\phi(\tfrac{x}{R})]\tfrac12|u|^4\,dx + \int[\psi(\tfrac{x}{R})+\phi(\tfrac{x}{R})]\tfrac23|u|^6\,dx. 
\]
We now use \eqref{weight-ids2} and radial Sobolev embedding to obtain
\begin{align*}
\int [\psi(\tfrac{x}{R})+\phi(\tfrac{x}{R})]|u|^4\,dx & = 2\int | u|^4\,dx + \mathcal{O}\biggl[ \int_{|x|>R} |u|^4\,dx \biggr] \\
& = 2\int | u|^4\,dx + \mathcal{O}\bigl[R^{-1} \| |x|^{\frac12}u\|_{L^\infty}^2 \|u\|_{L^2}^2\bigr] \\
& = 2\int |u|^4\,dx + \mathcal{O}[R^{-1}],
\end{align*}
and similarly
\[
\int [\psi(\tfrac{x}{R})+\phi(\tfrac{x}{R})]|u|^6\,dx = 2\int |u|^6\,dx + \mathcal{O}[R^{-2}]. 
\]

We therefore deduce that
\[
\tfrac{dA}{dt} \geq 2\int \phi(\tfrac{x}{R})|\nabla u|^2-\tfrac12|u|^4+\tfrac23|u|^6\,dx, 
\]
and the estimate \eqref{MV} now follows from the fundamental theorem of calculus.

We turn to \eqref{evacuation} and let $C>0$.  We now choose a sequence $T_n\to\infty$. In light of Lemma~\ref{L:bound-lambda}, we may find $\tilde C>0$ large enough that 
\begin{equation}\label{RnLB}
C\lambda(t_n)\leq R_n:=\tilde C T_n^{\frac12} \qtq{for all}t_n\in[0,T_n].
\end{equation}

Applying \eqref{MV}, we find
\[
\tfrac{1}{T_n}\int_{\frac12T_n}^{T_n} \int \chi_{R_n} \tfrac12 |\nabla u|^2 - \tfrac14 |u|^4 + \tfrac16 |u|^6\,dx \,dt \lesssim T_n^{-\frac12}.
\]
Thus we may find $t_n\in[\tfrac12T_n,T_n]$ so that
\[
\limsup_{n\to\infty} \int \chi_{R_n} \tfrac12|\nabla u(t_n,x)|^2 - \tfrac14|u(t_n,x)|^4+\tfrac16|u(t_n,x)|^6\,dx\leq 0. 
\]
Recalling \eqref{RnLB}, we obtain \eqref{evacuation}.\end{proof}

%%%%%%%%%%%%%%%%%%
\section{Localization of kinetic energy}
\label{S:localization}

Throughout this section, we suppose that $u:[0,\infty)\times\R^2\to\C$ is a solution to \eqref{nls} as in Proposition~\ref{P:compact}.  In particular, $u$ is parametrized by its scale function $\lambda(t)$.  We will denote the nonlinearity in \eqref{nls} by
\[
F(u) = -|u|^2 u + |u|^4 u. 
\]
The goal of this section is to prove the following: 

\begin{proposition}[Localization of kinetic energy]\label{P:localization} For any $\eta>0$, there exists $C(\eta)>0$ such that
\[
\sup_{t\in[0,\infty)}\int_{|x|>C(\eta)\lambda(t)}|\nabla u(t,x)|^2\,dx < \eta. 
\]
\end{proposition}

To prove Proposition~\ref{P:localization}, we argue essentially as in \cite{LiZhang, KLVZ}.  These works studied solutions the mass-critical NLS with mass equal to that of the ground state.  In our setting, some simplifications arise due to the fact that $u$ is known to be bounded in $H^1$.  The main ingredient in the proof is a frequency decay estimate (Lemma~\ref{L:frequency-decay}), which exhibits quantitative decay in frequency for the nonlinear part of the solution.  This in turn relies on the reduced Duhamel formula for $u$ (Corollary~\ref{C:reduced-duhamel}), along with the `in/out' decomposition for radial functions.

We begin by recording the following `mismatch estimates' as in \cite{KLVZ, LiZhang}, which will be used several times below.

\begin{lemma}[Mismatch estimates, \cite{KLVZ, LiZhang}]\label{L:mismatch} \text{ }
\begin{itemize}
\item For any $R\geq 1$ and $N>0$, we have
\[
\| \chi_R^c \nabla P_{\leq N} \chi_{R/2} \|_{L^2\to L^2} \lesssim_m N(NR)^{-m}\qtq{for any}m\geq 0,
\]
with the same bound if we replace $P_{\leq N}$ by $P_N$.
\item For any $R\geq 1$ and $N>0$, we have
\[
\| P_{\leq N}\chi_R^c P_{>4N}\|_{L^2\to L^2}\lesssim_m (NR)^{-m}\qtq{for any}m\geq 0,
\]
with the same bound if we replace $P_{>4N}$ by $P_{4N}$.  We may also insert a gradient in front of $P_{\leq N}$ or $P_{>4N}$, provided we change the bound to $N(NR)^{-m}$ and restrict to $m\geq 2$ for the $P_{>4N}$ estimate. 
\end{itemize}
\end{lemma} 

\begin{proof} These estimates may all be found in \cite{KLVZ, LiZhang} and are all based off of the principle non-stationary phase.  For the sake of completeness, let us demonstrate how to obtain one estimate of each type, say
\begin{equation}\label{mme1}
\| \chi_R^c  P_{\leq N} \chi_{R/2} \|_{L^2\to L^2} \lesssim_m (NR)^{-m}\qtq{for} m\geq 0,
\end{equation}
and
\begin{equation}\label{mme2}
\|P_N\chi_R^c \nabla P_{>4N}\|_{L^2\to L^2}\lesssim N(NR)^{-m}\qtq{for} m\geq 2. 
\end{equation}

For the first estimate, we observe that since the convolution kernel of $P_{\leq N}$ is a Schwartz function, we have the kernel bounds
\[
|\chi_R^c P_{\leq N}\chi_{R/2}(x,y)| \lesssim_\ell N^{d-\ell} |x-y|^{-\ell}\mathbf{1}_{|x-y|>R/2}
\]
for any $\ell>0$.  Thus \eqref{mme1} follows from Young's convolution inequality.  We also observe that if $\nabla P_{\leq N}$ were present, then we could repeat the same argument, noting that the convolution kernel of  $\tfrac{\nabla}{N}P_{\leq N}$ is a Schwartz function.

We will obtain the second estimate essentially from the first.  We first claim that we may bound
\[
\|P_{\leq N} \tilde \chi_\rho \tfrac{\nabla}{M} P_M \|_{L^2\to L^2}\lesssim_m (M\rho)^{-m}
\]
for any $M\geq 4N$, where $\tilde\chi_\rho$ is a cutoff to $\rho\leq |x|\leq 2\rho$.  Indeed, if we apply Plancherel's theorem, then we are led to consider (the adjoint of) an operator of the same type we considered in \eqref{mme1}. Summing over dyadic $\rho\geq R$ then yields
\[
\|P_{\leq N} \chi_R^c \nabla P_M \|_{L^2\to L^2}\lesssim_m M(MR)^{-m},
\]
and summing again over dyadic $M\geq 4N$ yields \eqref{mme2} (provided $m\geq 2$, so that the total power of $M$ is negative).  \end{proof} 

We next recall a weighted radial Strichartz estimate (as in \cite{KTV, LiZhang}):

\begin{lemma}[Radial Strichartz estimate, \cite{KTV}]\label{L:radial-strichartz}  For radial $F:I\times\R^2\to\C$ and $t,t_0\in I$, we have
\[
\biggl\| \int_{t_0}^t e^{i(t-\tau)\Delta}F(\tau)\,d\tau\biggr\|_{L^2(\R^2)} \lesssim \| |x|^{-\frac12} F\|_{L_t^{\frac43}L_x^1(I\times\R^2)}.
\]
\end{lemma}

We utilize the `incoming/outgoing' decomposition for radial functions introduced in \cite{KTV}.  In particular, for a radial function $f:\R^2\to\C$, we let
\[
[P^{\pm}f](r) = \tfrac12(r) \pm \tfrac{i}{\pi}\int_0^\infty \frac{f(\rho)}{r^2-\rho^2}\,\rho\,d\rho.
\]
We call $P^+$ the projection onto outgoing spherical waves and $P^-$ the projection onto incoming spherical waves. For a dyadic number $N>0$, we let $P_N^\pm$ denote the composition $P^\pm P_N$.

We record the essential facts we need concerning $P^\pm$ in the following lemma, which appears in \cite{KTV}. 

\begin{lemma}[Properties of $P^\pm$, \cite{KTV}]\label{L:propsPpm} The operators $P^\pm$ are bounded on $L^2(\R^2)$, with $P^++P^-$ giving the projection from $L^2$ onto $L^2_{rad}$. Furthermore, if
\[
|x|\gtrsim N^{-1} \qtq{and} t\gtrsim N^{-2},
\]
then we have the kernel estimate
\[
|P_N^\pm e^{\mp it\Delta}(x,y)| \lesssim \begin{cases} (|x|\,|y|)^{-\frac12} |t|^{-\frac12} & |y|-|x|\sim Nt, \\ \\ 
\tfrac{N^2}{(N|x|)^{1/2}\langle N|y|\rangle^{1/2}}\langle N^2 t+N|x|-N|y|\rangle^{-100} & \text{otherwise}.\end{cases}
\]
\end{lemma}

With the preliminaries in place, we can now establish the crucial frequency decay estimate.  Before stating and proving the result, we remind the reader of the notation $\chi_R$, $\chi_R^c$ for spatial cutoffs introduced in Section~\ref{S:notation}. 

\begin{lemma}[Frequency decay]\label{L:frequency-decay} For all $N\geq 1$ and $t\geq 0$,
\begin{equation}\label{freq-decay}
\| \chi_1^c P_Nu(t)\|_{L^2} \lesssim \|P_N u_0\|_{L^2} + N^{-\frac{6}{5}}.
\end{equation}
\end{lemma}

\begin{proof} We begin by writing
\[
\chi_1^c P_N u(t) = \chi_1^c P_N^- u(t) + \chi_1^c P_N^+u(t)
\]
and using the standard Duhamel formula for $P_N^-u(t)$ and the reduced Duhamel formula for $P_N^+u(t)$ (see Corollary~\ref{C:reduced-duhamel}).  This and a change of variables leads to
\begin{align}
\chi_1^c P_Nu(t) &  = \chi_1^c P_N^-e^{it\Delta}u_0 -i\chi_1^c \int_0^t P_N^-e^{i\tau\Delta}F(u(t-\tau))\,d\tau \label{incoming1} \\
& \quad + i\chi_1^c \int_0^\infty P_N^+e^{-i\tau\Delta}F(u(t+\tau))\,d\tau \label{outgoing1},
\end{align}
where the final integral is interpreted as a weak $L^2$ limit.  The linear evolution term is controlled by the first term on the right-hand side of \eqref{freq-decay}, and hence it suffices to consider the integral terms.  We focus on treating the term in \eqref{outgoing1}, as the remaining term may be handled in the same fashion. 

We begin by splitting 
\begin{align}
\eqref{outgoing1} & = i\chi_1^c \int_0^{N^{-1}} P_N^+e^{-i\tau\Delta}F(u(t+\tau))\,d\tau \label{outgoing2} \\
& \quad + i\chi_1^c \int_{N^{-1}}^\infty P_N^+ e^{-i\tau\Delta}\chi_{\frac12N\tau}F(u(t+\tau))\,d\tau \label{outgoing3}\\
& \quad + i\chi_1^c\int_{N^{-1}}^\infty P_N^+ e^{-i\tau\Delta}\chi_{\frac12N\tau}^cF(u(t+\tau))\,d\tau\label{outgoing4}.
\end{align}

We first estimate using Strichartz, H\"older, Bernstein, the fractional chain rule, and Sobolev embedding to obtain
\begin{align*}
\|\eqref{outgoing2}\|_{L^2} & \lesssim \|P_N F(u)\|_{L_\tau^1 L_x^2([t,t+N^{-1}]\times\R^2)} \\
& \lesssim N^{-\frac65}\| |\nabla|^{\frac15}[F(u)]\|_{L_t^\infty L_x^2} \\
& \lesssim N^{-\frac65}\bigl\{ \|u\|_{L_t^\infty L_x^6}^2\| |\nabla|^{\frac15}u\|_{L_t^\infty L_x^6}+\|u\|_{L_t^\infty L_x^{10}}^4 \| |\nabla|^{\frac15}u\|_{L_t^\infty L_x^{10}}\bigr\} \\
& \lesssim N^{-\frac65}\{\|u\|_{L_t^\infty H_x^1}^3+\|u\|_{L_t^\infty H_x^1}^5\} \lesssim N^{-\frac65}, 
\end{align*}
which is acceptable. 

We turn to \eqref{outgoing3}. We claim that by the kernel estimates in Lemma~\ref{L:propsPpm}, we have
\[
|[\chi_1^c P_N^+e^{-i\tau\Delta}\chi_{\frac12N\tau}](x,y)|\lesssim \frac{N^2}{(N^2\tau)^{50}\langle N|x-y|\rangle^{50}} \qtq{for} \tau\geq N^{-1}. 
\]
Indeed, it suffices to observe that we have the bounds
\[
\langle N^2\tau + N|x|-N|y|\rangle \gtrsim \max\{ N^2\tau,\langle N|x-y|\rangle\}
\]
in this regime. % For the N|x-y| term, if |x|<|y|/2, use the N^2\tau term to dominate, otherwise use the N|x| term to dominate.
Thus by Young's convolution inequality and Sobolev embedding, we obtain 
\begin{align*}
\|\eqref{outgoing3}\|_{L^2} & \lesssim \bigl\| \tfrac{N^2}{\langle N|x|\rangle^{50}}\ast F(u)\bigr\|_{L_t^\infty L_x^2} \cdot\int_{N^{-1}}^\infty (N^2\tau)^{-50}\,d\tau \\
& \lesssim N^{-51} \|F(u)\|_{L_t^\infty L_x^2} \\ 
& \lesssim N^{-51}\{\|u\|_{L_t^\infty L_x^6}^3 +\|u\|_{L_t^\infty L_x^{10}}^5\} \\
& \lesssim N^{-51}\{\|u\|_{L_t^\infty H_x^1}^3 + \|u\|_{L_t^\infty H_x^1}^5\} \lesssim N^{-51},
\end{align*}
which is acceptable. 

Finally, we turn to \eqref{outgoing4}. We begin by writing
\[
\chi_{\frac12 N\tau}^c F(u) = \chi_{\frac12 N\tau}^c F(\chi_{\frac14 N\tau}^c u) 
\]
and then further decomposing in frequency, leading to 
\begin{align}
\eqref{outgoing4} &= i\chi_1^c\int_{N^{-1}}^\infty P_N^+ e^{-i\tau\Delta} \chi_{\frac12 N\tau}^c P_{\leq \frac{N}{8}}F(\chi_{\frac14 N\tau}^c u(t+\tau))\,d\tau  \label{outgoing5} \\
& \quad + i\chi_1^c \int_{N^{-1}}^\infty P_N^+ e^{-it\Delta}\chi_{\frac12 N\tau}^c P_{>\frac{N}{8}}F(\chi_{\frac14 N\tau}^c u(t+\tau))\,d\tau.\label{outgoing6} 
\end{align}

To estimate \eqref{outgoing5}, we utilize the mismatch estimates in Lemma~\ref{L:mismatch}.  This yields
\begin{align*}
\|\eqref{outgoing5}\|_{L^2} & \lesssim \int_{N^{-1}}^\infty (N^2\tau)^{-50}\|F(u(t+\tau))\|_{L^2}\,d\tau \\
& \lesssim N^{-51} \bigl\{\|u\|_{L_t^\infty L_x^6}^3 + \|u\|_{L_t^\infty L_x^{10}}^5\bigr\} \lesssim N^{-51},
\end{align*}
which is acceptable.

Finally, we turn to \eqref{outgoing6}.  Observing that
\[
\|\chi_{\frac14 N\tau}^c u\|_{L_t^\infty \dot H_x^1} \lesssim \|u\|_{L_t^\infty H_x^1}\lesssim 1,
\]
we use Lemma~\ref{L:radial-strichartz}, Bernstein, the chain rule, radial Sobolev embedding to estimate
\begin{align*}
\|&\eqref{outgoing6}\|_{L^2}  \lesssim \| |x|^{-\frac12}\chi_{\frac12 N\tau}^c P_{>\frac{N}{8}}F(\chi_{\frac14 N\tau}^c u(\cdot+t))\|_{L_t^{\frac43}L_x^1([N^{-1},\infty))\times\R^2)} \\
& \lesssim \bigl\| (N\tau)^{-\frac12} N^{-1} \|\nabla F(\chi_{\frac14 N\tau}^cu(\cdot+t))\|_{L_x^1}\bigr\|_{L_t^{\frac43}([N^{-1},\infty))} \\
& \lesssim N^{-\frac32}\bigl\|\tau^{-\frac12} \|\chi_{\frac14 N\tau}^c u(\cdot+t)\|_{L_x^\infty}\bigr\|_{L_t^{\frac43}([N^{-1},\infty))}\|u\|_{L_t^\infty H_x^1}\{\|u\|_{L_t^\infty L_x^2}+\|u\|_{L_t^\infty L_x^6}^3\} \\
& \lesssim N^{-2} \| \tau^{-1}\|_{L_t^{\frac43}([N^{-1},\infty))} \| |x|^{\frac12} u\|_{L_{t,x}^\infty}\{\|u\|_{L_t^\infty H_x^1}^2 + \|u\|_{L_t^\infty H_x^1}^4\} \\
& \lesssim N^{-\frac{7}{4}}\{\|u\|_{L_t^\infty H_x^1}^3 + \|u\|_{L_t^\infty H_x^1}^5\} \lesssim N^{-\frac74},
\end{align*}
which is acceptable. \end{proof}

Finally, we turn to the proof of Proposition~\ref{P:localization}.  In fact, with Lemma~\ref{L:frequency-decay} in place, we can follow rather closely the proof of \cite[Theorem~1.11]{KLVZ}. 
%%%%%%%%%%%%%%%%%
\begin{proof}[Proof of Proposition~\ref{P:localization}]  Let $\eta>0$.  We then let $N_0>0$ be a large parameter and $\eta_1>0$ a small parameter to be chosen more precisely below.  Using compactness in $L^2$, we may find $C(\eta_1)$ sufficiently large so that
\begin{equation}\label{L2-tight}
\int_{|x|>C(\eta_1)\lambda(t)} |u(t,x)|^2\,dx < \eta_1 \qtq{for all}t\in[0,\infty). 
\end{equation}
Setting 
\[
R=2C(\eta_1)\lambda(t),
\]
we begin by estimating 
\begin{equation}\label{init-decomp}
\|\chi_R^c \nabla u(t)\|_{L^2} \leq \|P_{\leq N_0}\chi_R^c \nabla u(t)\|_{L^2}+\|P_{>N_0}\chi_R^c\nabla u(t)\|_{L^2}.
\end{equation}

\emph{Low frequencies.} We first estimate the low frequency term, beginning with
\begin{align}
\|&P_{\leq N_0}\chi_R^c \nabla u\|_{L^2}\nonumber \\
& \leq \|\chi_R^c \nabla P_{\leq 4N_0}\chi_{R/2} u\|_{L^2} + \|\nabla P_{\leq 4N_0}\chi_{R/2}^c u\|_{L^2} + \|P_{\leq N_0} \chi_R^c \nabla P_{>4N_0}u\|_{L^2}. \label{lowfreq1}
\end{align}

The first term on the right-hand side of \eqref{lowfreq1} is estimated by the mismatch estimates in Lemma~\ref{L:mismatch}, yielding the bound
\[
 \|\chi_R^c \nabla P_{\leq 4N_0}\chi_{R/2} u\|_{L^2} \lesssim R^{-2}N_0^{-1}\lesssim N_0^{-1} 
\]
For the second term on the right-hand side of \eqref{lowfreq1}, we instead use \eqref{L2-tight} and obtain
\[
\|\nabla P_{\leq 4N_0}\chi_{R/2}^c u\|_{L^2} \lesssim N_0 \eta_1.
\]
Finally, for the third term on the right-hand side of \eqref{lowfreq1} we use the mismatch estimates in Lemma~\ref{L:mismatch} to obtain
\[
\|P_{\leq N_0} \chi_R^c \nabla P_{>4N_0}u\|_{L^2}  \lesssim N_0(N_0 R)^{-2} \lesssim N_0^{-1}. 
\]
Putting together the pieces, we obtain
\[
\|P_{\leq N_0}\chi_R^c\nabla u\|_{L^2} \lesssim N_0^{-1}+\eta N_0. 
\]

\emph{High frequencies.}  For the high frequency term in \eqref{init-decomp}, we begin by writing
\begin{equation}\label{highfreq1}
\begin{aligned}
\| &P_{>N_0}\chi_R^c \nabla u\|_{L^2}^2 \\
& \lesssim \sum_{N>N_0}\|P_N\chi_R^c \nabla[P_{\leq N/4}+P_{>4N}]u\|_{L^2}^2 +\sum_{N>N_0}\|\chi_R^c \nabla P_{N/4\leq \cdot\leq 4N}u\|_{L^2}^2.
\end{aligned}
\end{equation}

The first term on the right-hand side of \eqref{highfreq1} can be estimated using the mismatch estimates of Lemma~\ref{L:mismatch}.  This yields
\[
\sum_{N>N_0}\|P_N\chi_R^c \nabla[P_{\leq N/4}+P_{>4N}]u\|_{L^2}^2  \lesssim \sum_{N>N_0} N^{-2}R^{-4} \lesssim N_0^{-2}.
\]
For the second term on the right-hand side of \eqref{highfreq1}, we further decompose as
\begin{equation}\label{highfreq2}
\begin{aligned}
\sum_{N>N_0}&\|\chi_R^c \nabla P_{N/4\leq \cdot\leq 4N}u\|_{L^2}^2 \\
& \lesssim \sum_{N>N_0/4}\|\chi_R^c \nabla\tilde P_N \chi_{R/2} P_N u\|_{L^2}^2 + \sum_{N>N_0/4} N^2\|  \chi_{R/2}^c P_N u\|_{L^2}^2.
\end{aligned}
\end{equation}
The first term on the right-hand side of \eqref{highfreq2} is amenable to the mismatch estimate in Lemma~\ref{L:mismatch}.  In particular, 
\[
\sum_{N>N_0/4}\|\chi_R^c \nabla\tilde P_N \chi_{R/2} P_N u\|_{L^2}^2 \lesssim \sum_{N>N_0/4} N^{-2}R^{-4} \lesssim N_0^{-2}. 
\]
Finally, for the second term on the right-hand side of \eqref{highfreq2}, we appeal to the frequency decay estimate, Lemma~\ref{L:frequency-decay} to obtain
\begin{align*}
\sum_{N>N_0/4} N^2 \|\chi_{R/2}^c P_N u\|_{L^2}^2 & \lesssim \sum_{N>N_0/4} N^2 \| P_N u_0\|_{L^2}^2 + \sum_{N>N_0/4} N^{-2/5} \\
&  \lesssim \|P_{>N_0/4}\nabla u_0\|_{L^2}^2 + N_0^{-2/5}. 
\end{align*}

Putting together all the pieces (including the low frequencies), we obtain
\[
\|\chi_R^c \nabla u(t)\|_{L^2} \lesssim \|P_{>N_0/4}\nabla u_0\|_{L^2} + N_0^{-1/5} + \eta_1 N_0. 
\]
Thus, choosing $N_0=N_0(\eta,u_0)$ sufficiently large and $\eta_1=\eta_1(N_0,\eta)$ sufficiently small, we obtain the desired estimate
\[
\|\chi_R^c \nabla u(t)\|_{L^2} \leq \eta \qtq{for all}t\geq 0,\qtq{where} R=2C(\eta_1)\lambda(t).
\]
\end{proof}

%%%%%%%%%%%%%%%%%
\section{Proof of the main result}\label{S:conclusion} 

Finally, we turn to the proof of Theorem~\ref{T}. 

\begin{proof}[Proof of Theorem~\ref{T}] If Theorem~\ref{T} fails, then we may find a solution $u$ as described in Proposition~\ref{P:compact}.  In particular, $u$ is pre-compact in $L^2$ modulo its scale function $\lambda(t)$.

Now let $\eta>0$. By Proposition~\ref{P:localization}, we may choose $C=C(\eta)>0$ large enough that 
\[
\sup_{t\in[0,\infty)} \int \tfrac12[1-\chi_{C\lambda(t)}] |\nabla u(t,x)|^2\,dx < \eta. 
\] 
On the other hand, appealing to Proposition~\ref{P:virial}, we may find a sequence $t_n\to\infty$ so that
\[
\limsup_{n\to\infty} \int \tfrac12 \chi_{C\lambda(t_n)}|\nabla u(t_n,x)|^2-\tfrac14|u(t_n,x)|^2+\tfrac16|u(t_n,x)|^6\,dx \leq 0. 
\]

Combining the previous two displays and using the conservation of energy, we obtain
\[
E(u) = \limsup_{n\to\infty} E(u(t_n)) \leq \eta. 
\]
As $\eta>0$ was arbitrary, we conclude $E(u)\leq 0$.  As this contradicts \eqref{positive-energy}, we complete the proof of Theorem~\ref{T}.\end{proof}

%%%


\begin{thebibliography}{100}

\bibitem{ADM} A. Arora, B. Dodson, and J. Murphy, \emph{Scattering below the ground state for the 2d radial nonlinear Schr\"odinger equation.} Proc. Amer. Math. Soc. \textbf{148} (2020), no. 4, 1653--1663.

\bibitem{CarlesSparber} R. Carles and C. Sparber, \emph{Orbital stability versus scattering in the cubic-quintic Schr\"odinger equation}. Rev. Math. Phys. \textbf{33} (2021), 2150004, 27pp.

\bibitem{Cheng} X. Cheng, \emph{Scattering for the mass super-critical perturbations of the mass critical nonlinear Schr\"odinger equations.} Illinois J. Math. \textbf{64} (2020), no. 1, 21--48.

\bibitem{Dodson} B. Dodson, \emph{Global well-posedness and scattering for the mass critical nonlinear Schr\"odinger
equation with mass below the mass of the ground state,} Adv. Math. \textbf{285} (2015), 1589--
1618.

\bibitem{DodMur} B. Dodson and J. Murphy, \emph{A new proof of scattering below the ground state for the 3d radial focusing cubic NLS.} Proc. Amer. Math. Soc. \textbf{145} (2017), no. 11, 4859--4867.

\bibitem{DLR} T. Duyckaerts, O. Landoulsi, and S. Roudenko, \emph{Threshold solutions in the focusing 3D cubic NLS equation outside a strictly convex obstacle}. Preprint {\tt arXiv:2010.07724}.

\bibitem{DM} T. Duyckaerts and F. Merle, \emph{Dynamic of threshold solutions for energy-critical NLS.} Geom. Funct. Anal. \textbf{18} (2009), no. 6, 1787--1840. 

\bibitem{DR} T. Duyckaerts and S. Roudenko, \emph{Threshold solutions for the focusing 3D cubic Schr\"odinger equation.} Rev. Mat. Iberoam. \textbf{26} (2010), no. 1, 1--56. 

\bibitem{GinibreVelo} J. Ginibre and G. Velo, \emph{On a class of nonlinear Schr\"odinger equations. II. Scattering theory, general case.} J. Functional Analysis \textbf{32} (1979), no. 1, 33--71.

\bibitem{KLVZ} R. Killip, D. Li, M. Visan, and X. Zhang, \emph{Characterization of minimal-mass blowup solutions to the focusing mass-critical NLS.} SIAM J. Math. Anal. \textbf{41} (2009), no. 1, 219--236.

\bibitem{KMV} R. Killip, J. Murphy, and M. Visan, \emph{Scattering for the cubic-quintic NLS: crossing the virial threshold.} Preprint {\tt arXiv:2007.07406}.

\bibitem{KOPV0} R. Killip, T. Oh, O. Pocovnicu, and M. Visan, \emph{Global well-posedness of the Gross-Pitaevskii and cubic-quintic nonlinear Schr\"odinger equations with non-vanishing boundary conditions.}  Math. Res. Lett. \textbf{19} (2012), no. 5, 969--986.

\bibitem{KOPV} R. Killip, T. Oh, O. Pocovnicu, and M. Visan, \emph{Solitons and scattering for the cubic-quintic nonlinear Schr\"odinger equation on $\mathbb{R}^3$}. Arch. Ration. Mech. Anal. \textbf{225} (2017), no. 1, 469--548.

\bibitem{KTV} R. Killip, T. Tao, and M. Visan, \emph{The cubic nonlinear Schr\"odinger equation in two dimensions with radial data.} J. Eur. Math. Soc. (JEMS) \textbf{11} (2009), no. 6, 1203--1258.

\bibitem{KVClay} R. Killip and M. Visan, \emph{Nonlinear Schr\"odinger equations at critical regularity}. In ``Evolution Equations'', 325--437, Clay Math. Proc., \textbf{17}, Amer. Math. Soc., Providence, RI, 2013.

\bibitem{LeCozTsai} S. Le Coz and T.-P. Tsai, \emph{Infinite soliton and kink-soliton trains for nonlinear Schr\"odinger equations.} Nonlinearity \textbf{27} (2014), no. 11, 2689--2709. 

\bibitem{LiZhang} D. Li and X. Zhang, \emph{On the rigidity of solitary waves for the focusing mass-critical NLS in dimensions $d\geq 2$.} Sci. China Math. \textbf{55} (2012), no. 2, 385--434.

\bibitem{Merle} F. Merle, \emph{Determination of blow-up solutions with minimal mass for nonlinear Schr\"odinger equations with critical power,} Duke Math. J. \textbf{69} (1993), 427--454.

\bibitem{MerleVega} F. Merle and L. Vega, \emph{Compactness at blow-up time for $L^2$ solutions of the critical nonlinear Schr\"odinger equation in $2D$.} Int. Math. Res. Not. \textbf{8} (1998), 399--425. 

\bibitem{MMZ} C. Miao, J. Murphy, and J. Zheng, \emph{Threshold scattering for the focusing NLS with a repulsive potential}. Preprint {\tt arXiv:2102.07163.} 

\bibitem{Strauss} W. A. Strauss, \emph{Existence of solitary waves in higher dimensions,} Comm. Math. Phys. \textbf{55} (1977), no. 2, 149--162. 

\bibitem{TVZ-mass} T. Tao, M. Visan, and X. Zhang,  \emph{Minimal-mass blowup solutions of the mass-critical NLS.} Forum. Math. \textbf{20} (2008), no. 5, 881--919. 

\bibitem{TVZ} T. Tao, M. Visan, and X. Zhang, \emph{The nonlinear Schr\"odinger equation with combined power-type nonlinearities.} Comm. Partial Differential Equations \textbf{32} (2007), no. 7-9, 1281-1343. 

\bibitem{TsutsumiYajima} Y. Tsutsumi and K. Yajima, \emph{The asymptotic behavior of nonlinear Schr\"odinger equations.} Bull. Amer. Math. Soc. (N.S.) \textbf{11} (1984), no. 1, 186--188.

\bibitem{Weinstein} M. Weinstein, \emph{Nonlinear Schr\"odinger equations and sharp interpolation estimates}, Comm. Math. Phys. \textbf{87} (1983), 567--576. %MR0691044

\bibitem{YZZ} K. Yang, C. Zeng, and X. Zhang, \emph{Dynamics of threshold solutions for energy critical NLS with inverse square potential.} Preprint {\tt arXiv:2006.04321.}

\bibitem{Zhang} X. Zhang, \emph{On Cauchy problem of $3-D$ energy-critical Schr\"odinger equations with subcritical perturbations}, J. Differential Equations \textbf{230} (2006), no. 2, 422--445. 
\end{thebibliography}
\end{document}